\documentclass[12pt]{amsart}
\pdfoutput=1

 \usepackage{amssymb,amsmath,amsfonts,epsfig,latexsym}

 \newtheorem{theorem}{Theorem}[section]
\newtheorem{definition}[theorem]{Definition}
\newtheorem{proposition}[theorem]{Proposition}

\newtheorem*{FrobeniusTheorem}{Diagonal Splitting Theorem}

\theoremstyle{definition}
\newtheorem{remark}[theorem]{Remark}
\newtheorem{example}[theorem]{Example}

\def\Z{\ensuremath{\mathbb{Z}}}

\def\F{\ensuremath{\mathbb{F}}}

\def\R{\ensuremath{\mathbb{R}}}

\def\OO{\ensuremath{\mathcal{O}}}

\def\<{\ensuremath{\langle}}
\def\>{\ensuremath{\rangle}}

\DeclareMathOperator{\Hom}{Hom}

\DeclareMathOperator{\Tor}{Tor}

\begin{document}

\title[Lattice polytopes cut out by root systems]{Lattice polytopes cut out by root systems and the Koszul property}

\author[Payne]{Sam Payne}

\thanks{Supported by the Clay Mathematics Institute}

\begin{abstract}
We show that lattice polytopes cut out by root systems of classical type are normal and Koszul, generalizing a well-known result of Bruns, Gubeladze, and Trung in type $A$.  We prove similar results for Cayley sums of collections of polytopes whose Minkowski sums are cut out by root systems.  The proofs are based on a combinatorial characterization of diagonally split toric varieties.
\end{abstract}

\maketitle

\section{Introduction}

Let $\Phi$ be a root system, with root lattice $N$, in the real vector space $N_\R = N \otimes_\Z \R$.  Let $M = \Hom(N,\Z)$ be the dual lattice of the root lattice, and $M_\R = M \otimes_\Z \R$.  Say that a polytope $P$ in $M_\R$ is \emph{cut out} by $\Phi$ if its facet normals are spanned by a subset of the roots, and that $P$ is a \emph{lattice polytope} if its vertices are in $M$. A lattice polytope $P$ is \emph{normal} if every lattice point in $mP$ is the sum of $m$ lattice points in $P$, for all positive integers $m$.

Following \cite{BrunsGubeladzeTrung97}, to which we refer for background on polytopal semigroup rings, write $S_P$ for the subsemigroup of $M \times \Z$ generated by lattice points in $P \times \{1\}$, with $\Z[S_P]$ the associated semigroup ring.  Say that $P$ is \emph{Koszul} if there is a linear resolution of $\Z$ as a $\Z[S_P]$-module.

\begin{theorem} \label{main}
Let $\Phi$ be a root system each of whose irreducible summands is of type $A$, $B$, $C$, or $D$, and let $P$ be a lattice polytope cut out by $\Phi$.  Then $P$ is normal and Koszul.
\end{theorem}

\noindent If $\Phi$ is an irreducible root system of type $A$, then we can choose coordinates identifying the root lattice with $\Z^n$ such that the set of roots is $\{ \pm e_i, \pm (e_j - e_k) \}$, where $e_i$ is the $i$-th standard basis vector.  In this case, the theorem was proved by Bruns, Gubeladze, and Trung using regular unimodular triangulations and Gr\"obner basis techniques \cite[Theorem~2.3.10]{BrunsGubeladzeTrung97}.  Our approach is based on what we call the Diagonal Splitting Theorem, a combinatorial result for lattice polytopes proved using characteristic free analogues of Frobenius splittings of toric varieties \cite{diagonalsplittings}.  See Section~\ref{preliminaries} for details.

\begin{remark}
Lattice polytopes cut out by $A_n$ and their triangulations were studied in unpublished work of Haase and Ziegler in the late 1990s and the work of Bruns, Gubeladze, and Trung.  More recently, Lam and Postnikov have made a detailed study of the triangulations of these polytopes induced by the affine Weyl arrangement \cite{LamPostnikov07}.  In the general setting of Theorem~\ref{main}, the affine Weyl arrangement of $\Phi$ gives a decomposition of $P$ into congruent products of simplices called \emph{alcoves}.  If the root system is irreducible then the alcoves are simplices, but these simplices do not have vertices in the dual lattice in general, even for $B_2$.  However, in some cases the set of vertices of alcoves is a lattice, and in these cases triangulations can be used to show that $P$ is normal with respect to the lattice of alcove vertices.  See  \cite{LamPostnikov08}.  It is not known whether all lattice polytopes cut out by root systems have regular unimodular triangulations.
\end{remark}

Theorem~\ref{main} generalizes to Cayley sums of lattice polytopes cut out by root systems as follows.  Let $P_1, \ldots, P_r$ be lattice polytopes in $M_\R$.  The \emph{Cayley sum} $P_1 * \cdots *P_r$ is the convex hull of $\big( P_1 \times \{e_1\} \big) \cup \cdots \cup \big(P_r \times \{e_r\} \big) $ in $M_\R \times \R^r$.  The Cayley sum is characterized by the property that projection to $\R^r$ maps $P_1 * \cdots * P_r$ onto the standard unimodular $(r-1)$-dimensional simplex, the convex hull of the standard basis vectors, and $P_i$ is the preimage of the $i$-th basis vector.  Recall that the \emph{Minkowski sum} $P_1 + \cdots + P_r$ is the set of all sums $x_1 + \cdots + x_r$, with $x_i \in P_i$.  

\begin{theorem} \label{Cayley}
Let $\Phi$ be a root system each of whose irreducible summands is of type $A$, $B$, $C$, or $D$.  Let $P_1, \ldots, P_r$ be lattice polytopes in $M_\R$ such that $P_1 + \cdots + P_r$ is cut out by $\Phi$.  Then $P_1 * \cdots * P_r$ is normal and Koszul.
\end{theorem}

\noindent Theorem~\ref{main} is the special case of Theorem~\ref{Cayley} when $r$ is equal to one.  It is not known whether all lattice polytopes satisfying the hypotheses of Theorem~\ref{Cayley} have regular unimodular triangulations, even for type $A$, though triangulations of Cayley sums have been studied intensively and correspond to mixed subdivisions of Minkowski sums; see \cite{Sturmfels94, HuberSturmfels95, HuberRambauSantos00, Santos05} and references therein.

\begin{remark}
Cayley sums of lattice polytopes correspond to projectivized sums of nef line bundles on toric varieties, and have appeared prominently in recent work related to boundedness questions in toric mirror symmetry \cite{BatyrevNill07b, BatyrevNill07, HNP}.  Cayley sums of normal polytopes are not normal in general (see Example~\ref{not normal}), and we do not know of any straightforward way of deducing Theorem~\ref{Cayley} from Theorem~\ref{main}.  
\end{remark}

It is not clear whether lattice polytopes cut out by root systems of exceptional type are always normal.  The techniques used to prove Theorems~\ref{main} and \ref{Cayley} do not work for $F_4$, and lattice polygons cut out by $G_2$ are normal but not Koszul in general.  See Sections~\ref{F} and \ref{G}.

\begin{remark}
Another natural class of polytopes associated to root systems are those with vertices in the root lattice whose edges are parallel to roots.  Recall that the Weyl fan $\Sigma_\Phi$ is the complete fan whose walls are given by the arrangement of hyperplanes perpendicular to the roots of $\Phi$.  So lattice polytopes whose edges are parallel to roots are exactly those whose inner normal fans are refined by the Weyl fan.  In toric geometry, these polytopes correspond to ample line bundles on normalizations of torus orbit closures in complete homogeneous spaces $G/P$.  These toric varieties and their cohomology have been studied extensively \cite{Procesi90, FlaschkaHaine91, Stembridge92, Stembridge94, DolgachevLunts94, Klyachko95, Dabrowski96, CarrellKurth00, HaasThesis, CarrellKuttler03}.  Recently, Howard has shown that in type $A$ all such polytopes are normal \cite{Howard07}.  It is not known whether these polytopes are Koszul, nor whether they have regular unimodular triangulations.
\end{remark}

\noindent {\bf Acknowledgments.}  I am most grateful to W.~Bruns and J.~Gubeladze for helpful comments on an earlier draft of this note, and the referee for an important correction.  I also thank I.~Dolgachev, C.~Haase, B.~Howard, B.~Sturmfels, and L.~Williams for stimulating conversations, M.~Joswig for bringing the work of Lam and Postnikov to my attention, and T.~Lam for sharing a copy of their preprint \cite{LamPostnikov08}.

\section{Preliminaries} \label{preliminaries}

Let $P$ be a lattice polytope.  Then $P$ is normal if and only if the semigroup $S_P$ is saturated in $M \times \Z$.  If the lattice points in $P \times \{1 \}$ generate $M \times \Z$ as a group then $P$ is normal if and only if the semigroup ring $\Z[S_P]$ is integrally closed or, equivalently, $k[S_P]$ is integrally closed for all fields $k$.  The Cayley sum of two normal polytopes need not be normal, as the following example shows.

\begin{example} \label{not normal}
Let $M = \Z^3$, and let $P$ be the convex hull of the lattice points $0$, $(1,0,0)$, $(0,0,1)$, and $(1,2,1)$.  The lattice point $(1,1,1)$ in $2P$ cannot be written as the sum of two lattice points in $P$, so $P$ is not normal.  However, projection to the third factor maps $P$ onto the unit interval.  Therefore, $P$ is the Cayley sum of the preimages in $\R^2$ of zero and one, which are the edges $[0, (1,0)]$ and $[0, (1,2)]$, respectively.  In particular, $P$ is the Cayley sum of two normal polytopes, but is not itself normal.
\end{example}

\noindent We now briefly discuss some characterizations of the Koszul property for lattice polytopes.

Let $S_P$ be partially ordered by setting $x \leq z$ if there is some $y \in S_P$ such that $x + y = z$.  The interval $[x,z]$ is the partially ordered set of all $y$ such that $x \leq y \leq z$.  Recall that a graded $k$-algebra $R$ is \emph{Koszul} if there is a linear resolution of $k$ as an $R$-module. 

\begin{proposition} \label{Koszul conditions}
Let $P$ be a lattice polytope.  Then the following are equivalent.
\begin{enumerate}
\item The polytope $P$ is Koszul.
\item For every $x \in S_P$, the interval $[0,x]$ is Cohen-Macaulay over $\Z$.
\item For every $x \in S_P$ and every field $k$, the interval $[0,x]$ is Cohen-Macaulay over $k$.
\item For every field $k$, the algebra $k[S_P]$ is Koszul.
\end{enumerate}
\end{proposition}

\begin{proof}
The equivalence of (2) and (3) follows from the topological characterization of Cohen-Macaulayness and the Universal Coefficient Theorem.  The equivalence of (3) and (4) is well-known, and may be deduced from the characterization of Koszul algebras in terms of $\Tor$ vanishing and the identification of $\Tor^{k[S_P]}(k,k)$ with sums of reduced homology groups of intervals $[0,x]$ \cite{LaudalSletsjoe85}.  We now show that (1) is equivalent to (2).

 $(1) \Rightarrow (2)$.  Suppose $P$ is Koszul, so there is a linear resolution of $\Z$ as a $\Z[S_P]$-module.  After tensoring with $\Z$, all of the differentials vanish, and it follows that $\Tor_i^{\Z[S_P]}(\Z, \Z)_j$ vanishes for $i \neq j$.  Now $\Tor^{\Z[S_P]}(\Z, \Z)$ can also be computed using the bar resolution of $\Z$, as in \cite{PeevaReinerSturmfels98, HerzogReinerWelker98}, to give a natural identification
\[
\Tor_{i}^{\Z[S_P]}(\Z, \Z)_j \ \cong \bigoplus_{x \, \in \, S_P \cap (M \times \{j\})} \widetilde H_{i-2}([0,x], \Z).
\]
It follows that $\widetilde H_{i-2}([0,x], \Z)$ vanishes unless $x$ is in $M \times \{i\}$, and hence $[0,x]$ is Cohen-Macaulay over $\Z$.

$(2) \Rightarrow (1)$.  Suppose $[0,x]$ is Cohen-Macaulay over $\Z$ for all $x$ in $S_P$.  Then the identification above shows that $\Tor_{i}^{\Z[S_P]}(\Z, \Z)_j$ vanishes for $i \neq j$.  By \cite[Theorem~2.3]{Woodcock98}, it follows that the Koszul complex
\[
\cdots \longrightarrow \Z[S_P] \otimes_\Z \Tor_{2}^\Z[S_P](\Z, \Z)_2 \longrightarrow  Z[S_P] \otimes_\Z \Tor_{1}^\Z[S_P](\Z, \Z)_1 \longrightarrow \Z[S_P] \rightarrow \Z
\]
is a linear resolution of $\Z$, and hence $P$ is Koszul.
\end{proof}

\begin{remark} 
In \cite{BrunsGubeladzeTrung97}, $P$ is defined to be Koszul if $k[S_P]$ is Koszul for every field $k$.  Here we work over the integers and define $P$ to be Koszul if there is a linear resolution of $\Z$ as a $\Z[S_P]$-module, because the existence of such a resolution is what follows most naturally from the diagonal splitting arguments in \cite{diagonalsplittings}.  By Proposition~\ref{Koszul conditions}, the two definitions are equivalent.
\end{remark}

\begin{remark}  
Let $F$ be a face of $P$.  Whenever a lattice point in $mF$ is written as the sum of $m$ lattice points in $P$, each of the summands is in the face $F$.  Therefore, any face of a normal polytope is normal.  Similarly, if $x$ is a point in $S_F$, then the interval $[0,x]$ in $S_F$ is the same as the interval $[0,x]$ in $S_P$, so any face of a Koszul polytope is Koszul \cite[Proposition~1.3]{OHH}.  In particular, it follows from Theorem~\ref{main} that any face of a lattice polytope cut out by a root system of classical type is normal and Koszul.
\end{remark}

In general, normal polytopes are not necessarily Koszul, and Koszul polytopes need not be normal, as the following examples show.

\begin{example} \label{not Koszul}
Let $M$ be the quotient of $\Z^3$ by the diagonal sublattice $\Z$, and let $P$ be the convex hull of the images of the three standard basis vectors, which contains exactly four lattice points, the three vertices plus the origin at its barycenter, as shown.
\begin{center}
\includegraphics{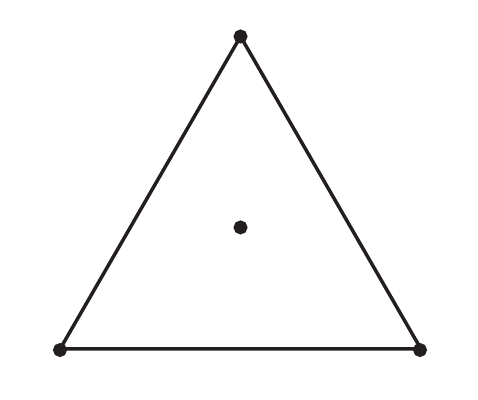}
\end{center}
Then $P$ is normal, but $\Z[S_P]$ is isomorphic to $\Z[x,y,z,w]/(xyz-w^3)$, which is not Koszul, since its ideal of relations is not generated in degree two. 
\end{example}

\begin{example}
Let $P$ be the polytope considered in Example~\ref{not normal}.  The semigroup ring $\Z[S_P]$ is isomorphic to a polynomial ring in four variables.  In particular, $P$ is Koszul but not normal.
\end{example}

One standard approach to proving that an algebra is Koszul is through Gr\"obner deformations; any homogeneous quotient of a polynomial ring that has a quadratic initial ideal with respect to some term order is Koszul \cite{BrunsHerzogVetter94, EisenbudReevesTotaro94}.  Square free initial ideals for $k[S_P]$ correspond to unimodular triangulations of $P$, and such an initial ideal is quadratic if and only if the corresponding triangulation is \emph{flag}, which means that every minimal nonface is an edge.

\begin{example}  
Let $P$ be the polytope considered in Example~\ref{not Koszul}, and let $k$ be a field.  Since $k[S_P]$ is not Koszul, it has no quadratic initial ideals, square free or not.  There is a unique unimodular triangulation of $P$, given by barycentric subdivision.  In this triangulation, $P$ is a minimal nonface, since each of its proper faces is a face of the triangulation, and the corresponding square free initial ideal for $k[S_P] \cong k[x,y,z,w]/(xyz-w^3)$ is $(xyz)$, which has degree three.
\end{example}

The approach to normality and the Koszul property that we follow here is based on a combinatorial characterization of diagonally split toric varieties.  Let $P$ be a lattice polytope, and let $v_1, \ldots, v_s \in N$ be the primitive generators of the inward normal rays of the facets of $P$.  We define the \emph{diagonal splitting polytope} $\F_P$ to be
\[
\F_P = \{ u \in M_\R \ | \ -1 \leq \<u, v_i\> \leq 1, \mbox{ for } 1 \leq i \leq s \}
\]
Note that $\F_P$ has nonempty interior and depends only on the facet normals of $P$.  The combinatorial type of $\F_P$ may be different from that of $P$.

\begin{definition} \label{split}
Let $q \geq 2$ be an integer.  A lattice polytope $P$ is diagonally split for $q$ if the interior of $\F_P$ contains representatives of every equivalence class in $\frac{1}{q} M / M$.
\end{definition}

\noindent Note that whether a lattice polytope is diagonally split depends only on $\F_P$.  In particular, diagonal splitting of a lattice polytope depends only on its facet normals.  The name ``diagonal splitting polytope" comes from the fact that if $X$ is the toric variety associated to the inner normal fan of $P$ then the diagonal is compatibly split in $X \times X$ with respect to $q$ if and only if the interior of $\F_P$ contains representatives of every equivalence class in $\frac{1}{q}M / M$.

\begin{FrobeniusTheorem}
Let $P_1, \ldots, P_r$ be lattice polytopes such that $P_1 + \cdots + P_r$ is diagonally split for some $q \geq 2$.  Then $P_1 * \cdots * P_r$ is normal and Koszul.
\end{FrobeniusTheorem}

\noindent In particular, if $P$ is diagonally split, then $P$ is normal and Koszul.  The Diagonal Splitting Theorem for lattice polytopes is a consequence of the theory of splittings of toric varieties, as we now explain.

Let $f: X \times X \rightarrow X \times X$ be the unique endomorphism whose restriction to the dense torus is given by $(t_1, t_2) \mapsto (t_1^q, t_2^q)$.  Then $f_* \OO_{X \times X}$ is naturally identified with the sheaf whose value on an invariant affine open $U_\sigma$ is $\Z[\sigma^\vee \cap \frac{1}{q}(M \times M)]$, with its natural module structure over the coordinate ring $\Z[U_\sigma] = \Z[\sigma^\vee \cap (M \times M)]$.  If $u \in \frac{1}{q} M$ is a fractional lattice point in the interior of $\F_P$ then there is a map of $\OO_{X \times X}$-modules $\pi_u : f_* \OO_{X \times X} \rightarrow \OO_{X \times X}$ given by
\[
\pi_u(x^{(u_1, \, u_2)}) = \left \{ \begin{array}{ll} x^{(u_1 - u, u_2 + u)} & \mbox{ if } (u_1 - u, u_2 + u) \mbox{ is in } (M \times M), \\ 0 & \mbox{ otherwise.} \end{array} \right.
\]
If $S$ is a set of representatives for every equivalence class in $\frac{1}{q} M / M$ in the interior of $\F_P$ then
$\pi = \sum_{u \in S} \pi_u$ is a splitting that is compatible with the diagonal, which means that $\pi \circ f^*$ is the identity on $\OO_{X \times X}$, and $\pi (f_* I_\Delta)$ is contained in $I_\Delta$, where $I_\Delta$ is the ideal of the diagonal subvariety.  Standard cohomology vanishing arguments from the theory of Frobenius splittings then show that every ample line bundle on $X$ is very ample and gives a projectively normal embedding \cite{BrionKumar05}.  For further details, and a proof that the homogeneous coordinate ring of each such embedding is Koszul, see \cite{diagonalsplittings}.

Here we apply the Diagonal Splitting Theorem combinatorially, with no further reference to the theory of toric varieties.

\section{Proofs of main results}

\begin{proposition} \label{type A}
Let $P$ be a lattice polytope cut out by the root system $A_n$.  Then $P$ is diagonally split for all $q \geq 2$.
\end{proposition}

\begin{proof}
We can choose coordinates identifying the root lattices of $A_n$ with $\Z^n$ so that the roots are given by
\[
A_n = \{ \pm e_i, \pm(e_j - e_k) \}.
\]
Then the dual lattice of the root lattice is also identified with $\Z^n$, and the interior of $\F_P$ contains the half open unit cube $[0,1)^n \subset \R^n,$ which contains representatives of every equivalence class in $\frac{1}{q} M / M$, as required.
\end{proof}

\begin{proposition} \label{type BCD}
Let $P$ be a lattice polytope cut out by the root system $B_n$, $C_n$, or $D_n$.  Then $P$ is diagonally split for all odd $q$.
\end{proposition}

\begin{proof}
Suppose $P$ is cut out by $B_n$.  We can choose coordinates identifying the root lattice of $B_n$ with $\Z^n$ so that the roots are given by
\[
B_n = \{ \pm e_i, \pm e_j \pm e_k \}.
\]
Then the dual lattice of the root lattice is also identified with $\Z^n$, and $\F_P$ contains the cube $[-\frac{1}{2}, \frac{1}{2}]^n$.  If $q$ is odd, then there are points in the interior of this cube representing every equivalence class in $\frac{1}{q}\Z^n/ \Z^n$, and hence $P$ is split.

Now, we can choose coordinates identifying the root lattices of $C_n$ and $D_n$ with the index two sublattice of $\Z^n$ consisting of points $(a_1, \ldots, a_n)$ such that $a_1 + \cdots + a_n$ is even, so that the root systems are given by
\[
C_n = \{ \pm 2e_i, \pm e_j \pm e_k \} \mbox{ \ \ and \ \ } D_n = \{ \pm e_i \pm e_j \}.
\]
In these coordinates, the roots of $D_n$ are a subset of those of $C_n$, so it will suffice to consider polytopes cut out by $C_n$.

Suppose $P$ is cut out by $C_n$.  In our chosen coordinates, the dual lattice of the root lattice is
\[
M = \Z^n + \Z \cdot (\textstyle{ \frac{1}{2}, \ldots, \frac{1}{2}} ).
\]
We claim that, if $q$ is odd, the points in $\frac{1}{q}\Z^n$ with coordinates of absolute value less than one half represent every equivalence class in $\frac{1}{q} M / M$.  To see the claim, note that if $u \in \frac{1}{q}M$ does not lie in $\frac{1}{q} \Z^n$, then $u$ is equivalent to $u + (\frac{1}{2}, \ldots, \frac{1}{2})$, which does lie in $\frac{1}{q}\Z^n$.  And any element of $\frac{1}{q}\Z^n$ is equivalent to one whose coordinates have absolute value less than one half, which proves the claim.  Now, any fractional lattice point in $\frac{1}{q}\Z^n$ whose coordinates have absolute value less than one half has inner product less than one with each root in $C_n$, and hence lies in the interior of  $\F_P$, so $P$ is diagonally split, as required.\end{proof}

\begin{proposition} \label{mixed diagonal}
Let $P$ be a polytope cut out by a root system $\Phi$ each of whose irreducible summands is of type $A$, $B$, $C$, or $D$.  Then $P$ is diagonally split for $q$ odd.
\end{proposition}

\begin{proof}
Let $\Phi_1, \ldots, \Phi_s$ be the irreducible summands of $\Phi$.  Let $N_i$ be the root lattice of $\Phi_i$, with $M_i$ the dual lattice.  Let $\F_i \subset M_i$ be the polytope cut out by the inequalties
\[
-1 \leq \<u,v\> \leq 1
\]
for each primitive generator $v$ of a facet normal of $P$ that is contained in $\Phi_i$.  By Propositions~\ref{type A} and \ref{type BCD}, the interior of $\F_i$ contains representatives of $\frac{1}{q}M_i / M_i$.  Furthermore, $\F_P$ contains the product $\F_1 \times \cdots \times \F_s$, by construction, and hence the interior of $\F_P$ contains the product of their interiors.  It follows that the interior of $\F_P$ contains representatives of every equivalence class in $\frac{1}{q}M / M$, as required.
\end{proof}

\noindent  We now apply Proposition~\ref{mixed diagonal} to prove our main result.

\begin{proof}[Proof of Theorem~\ref{Cayley}]
Let $\Phi$ be a root system each of whose irreducible summands is of type $A$, $B$, $C$, or $D$.  Suppose $P_1, \ldots, P_r$ are lattice polytopes in $M$ such that the Minkowski sum $P_1 + \cdots + P_r$ is cut out by $\Phi$.  Then $P_1 + \cdots + P_r$ is diagonally split for $q$ odd, by Proposition~\ref{mixed diagonal}.  Therefore $P_1 * \cdots *P_r$ is normal and Koszul, by the Diagonal Splitting Theorem.
\end{proof}

\section{Type $F$} \label{F}

The methods used in types $A$, $B$, $C$, and $D$ do not generalize in the most obvious possible way to type $F$.

\begin{proposition} \label{type F}
Let $P$ be a lattice polytope cut out by $F_4$ whose facet normals are exactly the rays spanned by the roots.  Then $P$ is not diagonally split.
\end{proposition}

\begin{proof}
We can choose coordinates identifying $N_\R$ with $\R^4$ such that the root lattice is
\[
N = \Z^4 + \Z \cdot (\textstyle{ \frac{1}{2}, \frac{1}{2}, \frac{1}{2}, \frac{1}{2} }),
\]
and the set of roots is $F_4 = \{ \pm e_i, \pm e_i \pm e_j, (\pm \frac{1}{2}, \pm \frac{1}{2}, \pm \frac{1}{2}, \pm \frac{1}{2}) \}$.  Then the dual lattice $M$ is the sublattice of $\Z^4$ consisting of points $(a_1, a_2, a_3, a_4)$ such that $a_1 + a_2 + a_3 + a_4$ is even.  We will show that the diagonal splitting polytope 
\[
\F = \{ u \in M_\R \ | \ -1 \leq \<u, v\> \leq 1 \mbox{ for all } v \in F_4 \}
\]
is too small in the sense that, for any $q \geq 2$, the interior of $\F$ contains strictly less than $q^4$ points in $\frac{1}{q}M$.  Therefore, the interior of $\F$ does not contain representatives for $\frac{1}{q}M/ M$.  The polytope $\F$ has rational vertices, and hence the number of fractional lattice points in $\F \cap \frac{1}{q}M$ is a quasipolynomial in $q$, called the \emph{Ehrhart quasipolynomial} of $\F$.  See \cite{BeckRobins07} for details and standard facts about Ehrhart quasipolynomials.

The set of vertices of the diagonal splitting polytope $\F$ is  $\{ \pm e_i, (\pm \frac{1}{2}, \pm \frac{1}{2}, \pm \frac{1}{2}, \pm \frac{1}{2}) \}$.  In particular, $2\F$ has vertices in $M$, so the period of the Ehrhart quasipolynomial of $\F$ is at most two.  In other words, if we set
\[
f(q) = \# \{ \F \cap {\textstyle \frac{1}{q}} M \},
\]
for positive integers $q$, then there are polynomials $f_0$ and $f_1$ such that $f(q)$ is equal to $f_0(q)$ if $q$ is even and $f_1(q)$ if $q$ is odd, for all positive integers $q$.  We then extend $f$ to negative integers by quasipolynomiality, and Ehrhart reciprocity says that $f(-q)$ is the number of points in $\frac{1}{q}M$  that lie in the interior of $\F$.  Dilating by a factor of $q$, we see that $f(q)$ is the number of lattice points in $q\F \cap M$ and, since $\F$ is cut out by inequalities of the form $\<u,v\> \leq 1$, the lattice points in the interior of $q\F$ are exactly the lattice points in $(q-1)\F$.  Therefore, we have
\[
f(-q) = f(q-1),
\]
for all integers $q$.  Since the degrees of $f_0$ and $f_1$ are both four, these polynomials are determined by the reciprocity formula above and the values $f(q)$ for $0 \leq q \leq 4$.  By counting lattice points, we find that
\begin{eqnarray*}
f(0) & = & 1. \\
f(1) & = & 1.  \\
f(2) & = & 49. \\
f(3) & = & 145. \\
f(4) & = & 433.
\end{eqnarray*}
It follows that the Ehrhart quasipolynomial of $\F$ is given by
\[
f(q) = \left\{ \begin{array}{ll} q^4 + 2q^3 + 2q^2 + 4q + 1 & \mbox{ for $q$ even}. \\
                                                 q^4 + 2q^3 + 2q^2 -2q -2 & \mbox{ for $q$ odd}.
                                                 \end{array} \right.
\]
In particular, we find that $f(-q)$ is strictly less than $q^4$ for all integers $q \geq 2$, and hence the interior of $\F$ cannot contain representatives of every equivalence class in $\frac{1}{q}M/ M$.
\end{proof}

\section{Type $G$} \label{G}

The analogue of Theorem~\ref{main} does not hold in type $G$; every lattice polygon is normal, but there is a lattice polygon cut out by $G_2$ that is not Koszul.  We can choose coordinates identifying the root lattice of $G_2$ with the rank two sublattice of $\Z^3$ consisting of points $(a_1, a_2, a_3)$ such that $a_1 + a_2 + a_3$ is zero and the set of roots is
\[
G_2 = \{ (e_i - e_j), \pm(e_i + e_j - 2e_k) \}.
\]
Then the dual lattice of the root lattice is the quotient of $\Z^3$ by the diagonal sublattice $\Z$.  Let $P$ be the polytope cut out by the inequalities $\<u,v\> \leq 1$ for roots $v$ of the form $e_i + e_j - 2 e_k$.  Then $P$ is the polytope considered in Example~\ref{not Koszul}, the convex hull of the images of the standard basis vectors in $\Z^3$.  In particular, $P$ is not Koszul.

\bibliography{math}

\providecommand{\bysame}{\leavevmode\hbox to3em{\hrulefill}\thinspace}
\providecommand{\MR}{\relax\ifhmode\unskip\space\fi MR }
\providecommand{\MRhref}[2]{%
  \href{http://www.ams.org/mathscinet-getitem?mr=#1}{#2}
}
\providecommand{\href}[2]{#2}
\begin{thebibliography}{HRW98}

\bibitem[BGT97]{BrunsGubeladzeTrung97}
W.~Bruns, J.~Gubeladze, and N.~Trung, \emph{Normal polytopes, triangulations,
  and {K}oszul algebras}, J. Reine Angew. Math. \textbf{485} (1997), 123--160.

\bibitem[BHV94]{BrunsHerzogVetter94}
W.~Bruns, J.~Herzog, and U.~Vetter, \emph{Syzygies and walks}, Commutative
  algebra (Trieste, 1992), World Sci. Publ., River Edge, NJ, 1994, pp.~36--57.

\bibitem[BK05]{BrionKumar05}
M.~Brion and S.~Kumar, \emph{Frobenius splitting methods in geometry and
  representation theory}, Progress in Mathematics, vol. 231, Birkh\"auser
  Boston Inc., Boston, MA, 2005.

\bibitem[BN07a]{BatyrevNill07b}
V.~Batyrev and B.~Nill, \emph{Combinatorial aspects of mirror symmetry},
  preprint, arXiv:math/0703456, 2007.

\bibitem[BN07b]{BatyrevNill07}
\bysame, \emph{Multiples of lattice polytopes without interior lattice points},
  Mosc. Math. J. \textbf{7} (2007), no.~2, 195--207, 349.

\bibitem[BR07]{BeckRobins07}
M.~Beck and S.~Robins, \emph{Computing the continuous discretely: Integer point
  enumeration in polyhedra}, Undergraduate Texts in Mathematics,
  Springer-Verlag, 2007.

\bibitem[CK00]{CarrellKurth00}
J.~Carrell and A.~Kurth, \emph{Normality of torus orbit closures in {$G/P$}},
  J. Algebra \textbf{233} (2000), no.~1, 122--134.

\bibitem[CK03]{CarrellKuttler03}
J.~Carrell and J.~Kuttler, \emph{Smooth points of {$T$}-stable varieties in
  {$G/B$} and the {P}eterson map}, Invent. Math. \textbf{151} (2003), no.~2,
  353--379.

\bibitem[Dab96]{Dabrowski96}
R.~Dabrowski, \emph{On normality of the closure of a generic torus orbit in
  {$G/P$}}, Pacific J. Math. \textbf{172} (1996), no.~2, 321--330.

\bibitem[DL94]{DolgachevLunts94}
I.~Dolgachev and V.~Lunts, \emph{A character formula for the representation of
  a {W}eyl group in the cohomology of the associated toric variety}, J. Algebra
  \textbf{168} (1994), no.~3, 741--772.

\bibitem[ERT94]{EisenbudReevesTotaro94}
D.~Eisenbud, A.~Reeves, and B.~Totaro, \emph{Initial ideals, {V}eronese
  subrings, and rates of algebras}, Adv. Math. \textbf{109} (1994), no.~2,
  168--187.

\bibitem[FH91]{FlaschkaHaine91}
H.~Flaschka and L.~Haine, \emph{Torus orbits in {$G/P$}}, Pacific J. Math.
  \textbf{149} (1991), no.~2, 251--292.

\bibitem[Haa02]{HaasThesis}
D.~Haas, \emph{A geometric study of the toric varieties determined by the root
  systems ${A}_n$, ${B}_n$ and ${C}_n$}, Ph.D. thesis, University of Michigan,
  2002.

\bibitem[HNP08]{HNP}
C.~Haase, B.~Nill, and S.~Payne, \emph{Cayley decompositions of lattice
  polytopes and upper bounds for $h^*$-polynomials}, preprint,
  arXiv:0804.3667v1, 2008.

\bibitem[How07]{Howard07}
B.~Howard, \emph{Matroids and geometric invariant theory of torus actions on
  flag spaces}, J. Algebra \textbf{312} (2007), no.~1, 527--541.

\bibitem[HRS00]{HuberRambauSantos00}
B.~Huber, J.~Rambau, and F.~Santos, \emph{The {C}ayley trick, lifting
  subdivisions and the {B}ohne-{D}ress theorem on zonotopal tilings}, J. Eur.
  Math. Soc. (JEMS) \textbf{2} (2000), no.~2, 179--198.

\bibitem[HRW98]{HerzogReinerWelker98}
J.~Herzog, V.~Reiner, and V.~Welker, \emph{The {K}oszul property in affine
  semigroup rings}, Pacific J. Math. \textbf{186} (1998), no.~1, 39--65.

\bibitem[HS95]{HuberSturmfels95}
B.~Huber and B.~Sturmfels, \emph{A polyhedral method for solving sparse
  polynomial systems}, Math. Comp. \textbf{64} (1995), no.~212, 1541--1555.

\bibitem[Kly95]{Klyachko95}
A.~Klyachko, \emph{Toric varieties and flag spaces}, Trudy Mat. Inst. Steklov.
  \textbf{208} (1995), no.~Teor. Chisel, Algebra i Algebr. Geom., 139--162.

\bibitem[LP07]{LamPostnikov07}
T.~Lam and A.~Postnikov, \emph{Alcoved polytopes {I}}, Discrete Comput. Geom.
  \textbf{38} (2007), 453--478.

\bibitem[LP08]{LamPostnikov08}
\bysame, \emph{Alcoved polytopes {II}}, in preparation, 2008.

\bibitem[LS85]{LaudalSletsjoe85}
O.~Laudal and A.~Sletsj{\o}e, \emph{Betti numbers of monoid algebras.
  {A}pplications to {$2$}-dimensional torus embeddings}, Math. Scand.
  \textbf{56} (1985), no.~2, 145--162.

\bibitem[OHH00]{OHH}
H.~Ohsugi, J.~Herzog, and T.~Hibi, \emph{Combinatorial pure subrings}, Osaka J.
  Math. \textbf{37} (2000), no.~3, 745--757.

\bibitem[Pay08]{diagonalsplittings}
S.~Payne, \emph{Frobenius splittings of toric varieties}, preprint,
  arXiv:0805.1252v1, 2008.

\bibitem[Pro90]{Procesi90}
C.~Procesi, \emph{The toric variety associated to {W}eyl chambers}, Mots, Lang.
  Raison. Calc., Herm\`es, Paris, 1990, pp.~153--161.

\bibitem[PRS98]{PeevaReinerSturmfels98}
I.~Peeva, V.~Reiner, and B.~Sturmfels, \emph{How to shell a monoid}, Math. Ann.
  \textbf{310} (1998), no.~2, 379--393.

\bibitem[San05]{Santos05}
F.~Santos, \emph{The {C}ayley trick and triangulations of products of
  simplices}, Integer points in polyhedra---geometry, number theory, algebra,
  optimization, Contemp. Math., vol. 374, Amer. Math. Soc., Providence, RI,
  2005, pp.~151--177.

\bibitem[Ste92]{Stembridge92}
J.~Stembridge, \emph{Eulerian numbers, tableaux, and the {B}etti numbers of a
  toric variety}, Discrete Math. \textbf{99} (1992), no.~1-3, 307--320.

\bibitem[Ste94]{Stembridge94}
\bysame, \emph{Some permutation representations of {W}eyl groups associated
  with the cohomology of toric varieties}, Adv. Math. \textbf{106} (1994),
  no.~2, 244--301.

\bibitem[Stu94]{Sturmfels94}
B.~Sturmfels, \emph{On the {N}ewton polytope of the resultant}, J. Algebraic
  Combin. \textbf{3} (1994), no.~2, 207--236.

\bibitem[Woo98]{Woodcock98}
D.~Woodcock, \emph{Cohen-{M}acaulay complexes and {K}oszul rings}, J. London
  Math. Soc. (2) \textbf{57} (1998), no.~2, 398--410.

\end{thebibliography}
\bibliographystyle{amsalpha}

\end{document}